\documentclass[runningheads]{llncs}

\usepackage{graphicx}
\usepackage{algorithm}
\usepackage{algorithmicx}
\usepackage{algpseudocode}
\usepackage{comment}
\usepackage{enumitem}
\usepackage{todonotes}
\algrenewcommand\algorithmicrequire{\textbf{Input:}}
\algrenewcommand\algorithmicensure{\textbf{Output:}}

\usepackage{amsthm}

\begin{document}
\title{On the Structure of Hamiltonian Graphs with Small Independence Number\thanks{Supported by GAUK 370122.} }
\titlerunning{On the Structure of Hamiltonian Graphs with Small Independence Number}
%
\author{Nikola Jedli\v{c}kov\'{a}\inst{1}\orcidID{0000-0001-9518-6386} \and
Jan Kratochvíl\inst{1}\orcidID{0000-0002-2620-6133}}
\authorrunning{N. Jedličková, J. Kratochvíl}
%
\institute{Department of Applied Mathematics, Faculty of Mathematics and Physics, Charles University, Prague, Czech Republic \\ \email{jedlickova@kam.mff.cuni.cz}, \email{honza@kam.mff.cuni.cz}}

\maketitle              

\begin{abstract}
A Hamiltonian path (cycle) in a graph is a path (cycle, respectively) which passes through all of  its vertices. The problems of deciding the existence of a Hamiltonian cycle (path) in an input graph are well known to be NP-complete, and restricted classes of graphs which allow for their polynomial-time solutions are intensively investigated.  Until very recently the complexity was open even for graphs of independence number at most 3. So far unpublished result of Jedli\v{c}kov\'{a} and Kratochv\'{\i}l [arXiv:2309.09228] shows that for every integer $k$, Hamiltonian path and cycle are polynomial-time solvable in graphs of independence number bounded by $k$.
As a companion structural result, we determine explicit obstacles for the existence of a Hamiltonian path  for small values of $k$, namely for graphs of independence number 2, 3, and 4. Identifying these obstacles in an input graph yields alternative polynomial-time algorithms for Hamiltonian path and cycle with no large hidden multiplicative constants.

\keywords{Graph \and Hamiltonian path \and Hamiltonian cycle \and Path Cover \and Independence number \and Polynomial-time algorithm.}
\end{abstract}

\section{Introduction}\label{sec:Intro}

A cycle in a graph is {\em Hamiltonian} if it contains all vertices of the graph. A graph is  called {\em Hamiltonian} if it contains a Hamiltonian cycle. The notion of Hamiltonian graphs is well-known and intensively studied in graph theory. Many sufficient conditions for Hamiltonicity of graphs are known (e.g., minimum degree at least $|V(G)|/2$ which goes back to Dirac~\cite{dirac1952some}, the theorem of Ore~\cite{ore1960note} whose short proof from \cite{bondy2003short} is one of the jewels of graph theory, Chv\'atal's conditions on the degree sequence~\cite{chvatal1972hamilton}), but a simple necessary and sufficient condition is not known (and not likely to exists, as the problem itself is NP-complete). 
On the other hand, many open questions and conjectures are around. The more than 50 years old conjecture of Barnette states that every cubic planar bipartite 3-connected graph is Hamiltonian, to mention at least one. 

Hamiltonian cycles in solution spaces of configurations of certain types are called {\em Gray codes}. A prominent example is that the $2^n$ bitvectors of length $n$ (i.e., the vertices of the $n$-dimensional hypercube) can be arranged in a cyclic order so that the Hamming distance (i.e., the number of coordinates in which they differ) of any two consecutive ones equals $1$. In this connection we have to mention the recent proof of the famous Middle-levels conjecture, stating that the middle level subgraph of the hypercube of odd dimension is Hamiltonian~\cite{Mutze2016}. M\"utze also proved another almost 50 years old conjecture on Hamiltonicity of odd Kneser graphs in~\cite{Mutze2018}.

A {\em Hamiltonian path} in a graph is a path that contains all vertices of the graph. Obviously, every Hamiltonian graph contains a Hamiltonian path, and a graph is Hamiltonian if it contains a Hamiltonian path connecting a pair of adjacent vertices. A graph is called {\em Hamiltonian connected} if every two distinct vertices are connected by a Hamiltonian path. Chv\'atal and Erd\H{o}s stated and proved elegant sufficient conditions for Hamiltonian connectedness and for the existence of Hamiltonian paths in terms of comparing the vertex connectivity and independence number of the graph under consideration, cf. Proposition~\ref{prop:chvatal}.
For an excellent survey on Hamiltonian graphs cf.~\cite{gould2003advances} and a more recent one~\cite{gould2014recent}.

From the computational complexity point of view, all these problems are hard. Karp~\cite{Karp1972} proved already in 1972 that deciding the existence of Hamiltonian paths and cycles in an input graph are NP-complete problems. In this sense the problem of deciding the existence of a Hamiltonian path that connects two given vertices is the canonical one - if this can be solved in polynomial time in a given graph (or graphs from a given graph class), then all the other problems (existence of a Hamiltonian path, existence of a Hamiltonian cycle, and Hamiltonian connectedness) can be solved in polynomial time by checking all pairs of (adjacent, in the case of Hamiltonian cycle) vertices of the input graph.    

In order to study the borderline between easy (polynomial-time solvable) and hard (NP-complete) variants of the problems, researchers have intensively studied the complexity of Hamiltonian-related problems in special graph classes. 

Deciding the existence of a Hamiltonian cycle remains NP-complete on planar graphs~\cite{garey1976planar}, circle graphs~\cite{damaschke1989hamiltonian} and for several other generalizations of interval graphs~\cite{bertossi1986hamiltonian}, for chordal bipartite graphs~\cite{mu96} and for split (and therefore also for chordal) graphs~\cite{golumbic2004algorithmic}, but is solvable in linear time in interval graphs~\cite{keil1985finding} and in convex bipartite graphs~\cite{mu96}.
The existence of a Hamiltonian path can be decided in polynomial time for cocomparability graphs~\cite{damaschke1991finding} and for circular arc graphs~\cite{damaschke1993paths}. 

Many of the above mentioned graph classes are hereditary, i.e., closed in the induced subgraph order, and as such can be described by collections of forbidden induced subgraphs. This has led to carefully examining $H$-free graphs, i.e., graph classes with a single forbidden induced subgraph. For many graphs $H$, the class of $H$-free graphs has nice structural properties (e.g., $P_4$-graphs are exactly the cographs) or their structural properties can be used to design polynomial algorithms for various graph-theoretical problems (as a recent example, cf. colouring of $P_6$-free graphs in~\cite{Spirkl2019}). For Hamiltonian-type problems, consider three-vertex graphs $H$. For $H=K_3$, both Hamiltonian cycle and Hamiltonian path are NP-complete on triangle-free graphs, since they are NP-complete on bipartite graphs~\cite{mu96}. For $H=3K_1$, the edgeless graph on three vertices, both Hamiltonian cycle and Hamiltonian path are polynomial time solvable~\cite{Duf1981}. The remaining two graphs, $P_3$ and $K_1+K_2$, are induced subgraphs of $P_4$, thus the corresponding class of $H$-free graphs is a subclass of cographs, and as such a subclass of cocomparability graphs, in which the Hamiltonian path and Hamiltonian cycle problems are solvable in polynomial time~\cite{damaschke1991finding}. However, a complete characterization of graphs $H$ for which Hamiltonian path or Hamiltonian cycle problems are solvable in polynomial time (and for which they are NP-complete) when restricted to $H$-free input graphs is not in sight. 

{\em Our results.}
Recently, there has been a substantial progress at least for the case when $H$ is an edge-less graph, i.e., $H=kK_1$ for some $k$. (Note the straightforward connection to the independence number of the input graphs -- a graph is $kK_1$-free if and only if its independence number does not exceed $k-1$.)  It is shown in~\cite{jedlickova2023hamiltonian} that Hamiltonian path and Hamiltonian cycle are polynomial-time solvable in graphs of independence number at most $k$, for every $k$, i.e., that these problems are in the class XP when parameterized by the independence number of the input graph. 
 The goal of this paper is to provide a companion structural description of feasible instances for small values of $k=2,3,4$. The connectivity level of the input graph plays a role in the case distinction of the characterization, and theorems of Chv\'atal and Erd\H{o}s are handy here (cf. Section~\ref{sec:Prelim}). For instance, it is relatively easy to describe the feasible instances of Hamiltonian path with prescribed end-vertices in a $3K_1$-free graph, say $G$ -- none of the prescribed end-vertices may be an articulation of the input graph, together they should not form a cut of size two, and if the graph is not vertex-2-connected, then for every articulation $x$, the prescribed end-vertices belong to different components of $G-x$ (cf.~Theorem~\ref{thm:3K1_uv}). All these conditions can be checked quickly, and thus these obstacles provide not only a structural result, but also a practical algorithm. For $4K_1$-free and $5K_1$-free graphs, the conditions become more complicated (cf. Theorems~\ref{thm:4k1},\ref{thm:5k1}).        

\section{Preliminaries} \label{sec:Prelim}

We consider simple undirected graphs without loops or multiple edges.  
The vertex set of a graph $G$ is denoted by $V(G)$, its edge set by $E(G)$. Edges are considered as two-element sets of vertices, thus we write $u\in e$ to express that a vertex $u$ is incident with an edge $e$.  For the sake of brevity, we write $uv$ instead of $\{u,v\}$ for the edge containing vertices $u$ and $v$. We say that $u$ is {\em adjacent to} $v$ if $uv\in E(G)$. The {\em degree} of a vertex is the number of other vertices adjacent to it. The subgraph of $G$ induced by vertices $A\subseteq V(G)$ will be denoted by $G[A]$. The {\em independence number}  of a graph $G$, denoted by $\alpha(G)$, is the order of the largest edgeless induced subgraph of $G$. With the standard notion of $K_k$ being the complete graph with $k$ vertices and $G+H$ being the disjoint union of graphs $G$ and $H$, $\alpha(G)$ is equal to the largest $k$ such that $kK_1$ is an induced subgraph of $G$. A graph is called {\em $H$-free} if it contains no induced subgraph isomorphic to $H$.  

A {\em path} in a graph $G$ is a sequence of distinct vertices such that any two consecutive ones are adjacent. A {\em cycle} is formed by a path connecting two adjacent vertices. The path (cycle) is {\em Hamiltonian} if it contains all vertices of the graph. The \emph{path cover number} of $G$, denoted by $\mathrm{pc}(G)$, is the smallest number of vertex disjoint paths that contain all vertices of $G$.

A graph is connected if any two vertices are connected by a path. Since the problems we are interested in are either trivially infeasible on disconnected graphs, or can be reduced to studying the components of connectivity one by one, we only consider connected input graphs in the sequel. 

A {\em vertex  cut} in a graph $G$ is a set $A\subset V(G)$ of vertices such that the graph $G-A=G[V(G)\setminus A]$ is disconnected. The {\em vertex connectivity} $c_v(G)$ of a graph $G$ is the order of a minimum cut in $G$, or $|V(G)|-1$ if $G$ is a complete graph. Since we will not consider edge connectivity, we will often omit the adjective when talking about the connectivity measure, we always have vertex connectivity in mind. 

Although we consider undirected graphs, when we talk about a path in a graph, the path itself is considered traversed in the direction from its starting vertex to the ending one. Formally, when we say that a path $P$ connects a vertex $x$ to a vertex $y$, then by $P^{-1}$ we denote the same path, but traversed from $y$ to $x$. This is important when creating a longer path by concatenating shorter ones.  

We will be using the following well-known corollary of a theorem of Menger.

\begin{proposition}\label{prop:menger}  \cite{denley2001generalization}
Let $G$ be an $s$-connected graph with $s \geq 1$. If $x \in V(G)$, $Y \subseteq V(G)$ and $x \not \in Y$, then there exist distinct vertices $y_1, \ldots, y_m \in Y$, where $m= \min \{s, |Y|\}$, and internally disjoint paths $P_1, \ldots, P_m$ such that for every $i \in 1,\ldots,m$,
\begin{itemize}
\item 
$P_i$ is a path starting in vertex $x$ and ending in vertex $y_i$, and 
\item 
$P_i \cap Y = \{y_i\}$.
\end{itemize}
\end{proposition}

We build upon the following results of Chv\'atal and Erd\H{o}s.

\begin{proposition} \cite{chvatal1972note} \label{prop:chvatal} 
Let $G$ be an $s$-connected graph. 
\begin{enumerate}
\item If $\alpha(G)<s+2$, then $G$ has a Hamiltonian path. 
\item If $\alpha(G)<s+1$, then $G$ has a Hamiltonian cycle.
\item If $\alpha(G)<s$, then $G$ is Hamiltonian connected (i.e. every pair of vertices is joined by a Hamiltonian path).
\end{enumerate}
\end{proposition}

\section{Hamiltonian paths in graphs with small independence number} \label{sec:smallcases}

In this section we explicitly describe the situations when Hamiltonian paths do not exist for $kK_1$-free graphs for small cases of $k=3,4,5$. The descriptions clearly yield polynomial time decision algorithms with no large hidden multiplicative constants in the running time estimate.    

For the description, we will use the following notation.
A path cover of size 2 with one path starting in a vertex $u$ and the other one starting in a vertex $v$ will be denoted by $\mathrm{PC}(u,v)$. A path $P$ with endpoints $u,v$ will be denoted by $P_{u,v}$. Furthermore, $c(G)$ will denote the number of connected components of $G$. Note that a vertex cut of size 1 is called an {\em articulation point}. The set of articulation points in a graph $G$ will be denoted by $Art(G)$. 

\subsection{$3K_1$-free graphs}\label{sec:3k1}

It follows from Proposition~\ref{prop:chvatal}.1 that each connected $3K_1$-free graph has a Hamiltonian path. We will establish when a $3K_1$-free graph has a Hamiltonian path, resp. path cover starting in specified vertex (resp. vertices).

\begin{theorem}\label{thm:3K1_uv}
Let $G$ be a connected $3K_1$-free graph and let $u,v$ be distinct vertices of $G$. Then $G$ has a Hamiltonian path $P_{u,v}$ if and only if
\begin{enumerate}[label=(\alph*)]
\item none of the vertices $u$ and $v$ is an articulation point in $G$, and 
\item for every articulation point $x$, the vertices $u$ and $v$ belong to different components of $G-\{x\}$, and 
\item $\{u,v\}$ is not a minimum vertex cut in $G$,
\end{enumerate}
which happens if and only if $G$ satisfies (a), (b) and\\[5pt]
{(c')} $\{u,v\}$ is not a vertex cut of size two in $G$.
\end{theorem}

\begin{proof} We will first prove that (c') is (under the assumption of (a) and (b)) equivalent to (c).
Clearly, if $\{u,v\}$ is not a vertex cut in $G$, then it is not a minimum vertex cut, and thus (c') implies (c). For the opposite implication, suppose that (a), (b) and (c) are satisfied and (c') is violated, i.e., $\{u,v\}$ is a vertex cut in $G$. Then $\{u,v\}$ is a minimal vertex cut, otherwise $u$ or $v$ is an articulation point which is forbidden by condition (a). Now (c) implies that there is a smaller cut, i.e., an articulation point, say $x$, in $G$.  By (a), $x \neq u,v$,  and by (b), $u,v$ are in different components of $G - \{x\}$. Both of these two components are cliques since $G$ is $3K_1$-free. Then $\{u,v\}$ is a minimal vertex cut if and only if $u,v$ are the only neighbours of $x$ and each of the components of $G - \{x\}$ contains at least two vertices. But then both $u,v$ are articulation points, a contradiction with (a).  

Secondly, if at least one of the conditions (a)-(c') is violated, then $G$ does not have a Hamiltonian path $P_{u,v}$.

Finally,
assume that all of the conditions (a)-(c) are satisfied. We will show that $G$ has a Hamiltonian path $P_{u,v}$.

\emph{Case 1. $G$ is connected but not 2-connected.} There exists an articulation point $x$. Since $G$ is $3K_1$-free, $G-\{x\}$ has exactly 2 connected components. Let us denote them by $Q_1, Q_2$. Since $G$ is $3K_1$-free, both of the components are cliques, $x$ is adjacent to all vertices of at least one of them, say $Q_2$, and $x$ is adjacent to at least one vertex, say $x'$, in the other clique $Q_1$. 

The vertex $x$ is different from $u,v$ by (a) and $u,v$ are in different components of $G-\{x\}$ by (b), say $u \in Q_1$ and $v \in Q_2$. If $u=x'$ and $|Q_1|>1$, there is another vertex $x''\in Q_1$ adjacent to $x$ because $u$ is not an articulation point. We obtain a Hamiltonian path $P_{u,v}$ by concatenating a path in $G[Q_1]$ starting in vertex $u$, passing through all vertices of $Q_1$ and ending in $x'$ (or in $x''$, if $u=x'$ and $|Q_1|>1$) with a path $P_{x,v}$ in $G[Q_2\cup \{x\}]$ passing through all vertices of $Q_2$ and ending in vertex $v$.

\emph{Case 2. $G$ is $2$-connected but not $3$-connected.} We will use a similar approach as in \cite{chvatal1972note}. Take a longest path $P$ from $u$ to $v$ and suppose for a contradiction that this is not a Hamiltonian path. Then there exists a vertex $x \in V(G) \setminus V(P)$ and since $G$ is 2-connected, $P$ contains at least 3 vertices. By Proposition~\ref{prop:menger}, there exist vertices $y_1,y_2$ (ordered in the direction from $u$ to $v$) on $P$ together with internally disjoint paths $P_1,P_2$ from $x$ to $y_1,y_2$. If $y_1,y_2$ are consecutive on $P$, then we obtain a longer path by replacing the edge $y_1y_2$ by the concatenation of the paths $P^{-1}_1$ and $P_2$, contradicting the maximality of $P$. Thus we assume that no two disjoint paths from $x$ to $P$ end in consecutive vertices of $P$.

If $y_1=u$ and $y_2=v$ and there is no path from $x$ to $V(P)$ ending in a vertex different from $u$ and $v$, then $\{u,v\}$ is a vertex cut in $G$. Since $G$ is 2-connected, it is also a minimum vertex cut violating condition (c). Hence we may choose $y_1,y_2$ so that at least one of them is not an end-vertex of $P$. Assume without loss of generality that $y_2 \neq v$ (the other case is symmetric). Denote by $y_1'$ the successor of $y_1$ and by $y_2'$ the successor of $y_2$ on the path $P$ (in the direction from $u$ to $v$). Then $xy_1', xy_2'$ are not edges in $G$, otherwise we would have a longer path. Since $G$ is $3K_1$-free, $y_1'y_2'$ is an edge. But then the path obtained by concatenating the subpath of $P$ from $u$ to $y_1$, $P_1^{-1}$, $P_2$, the subpath of $P^{-1}$ from $y_2$ to $y_1'$, the edge $y_1'y_2'$ and the subpath of $P$ from $y_2'$ to $v$, is a path from $u$ to $v$ longer than $P$, a contradiction. 

\emph{Case 3. $G$ is at least 3-connected.} In this case, $G$ is Hamiltonian connected by Proposition~\ref{prop:chvatal}. 
\end{proof}

\begin{theorem} \label{thm:3K1u}
Let $G$ be a connected $3K_1$-free graph. Then $G$ has a Hamiltonian path starting at $u \in V(G)$ if and only if $u$ is not an articulation point in $G$.
\end{theorem}
\begin{proof}
If $u$ is an articulation point then such Hamiltonian path does not exist. 

Suppose that $u$ is not an articulation point and that $G$ is connected but not 2-connected. There is an articulation point $x$ in $G$ and $x \neq u$. The graph $G-\{x\}$ has exactly two components and they are cliques. We construct a Hamiltonian path starting at vertex $u$, going through all vertices of the clique containing $u$ to a neighbour of $x$ (note that either $u$ is the only vertex in its component, or $x$ is adjacent to another vertex in this component, since $u$ would be an articulation point otherwise), moving through $x$ to the other component and then traversing all the vertices of the other clique in an arbitrary order.

If $G$ is at least 2-connected, then it has a Hamiltonian cycle by Proposition~\ref{prop:chvatal}, thus it has a Hamiltonian path starting at any vertex.
\end{proof}

\begin{theorem} \label{thm:3K1pathcover}
Let $G$ be a connected $3K_1$-free graph. Then $G$ has a path cover $\mathrm{PC}(u,v)$ for every two distinct $u,v \in V(G)$.
\end{theorem}
\begin{proof}
If $G$ has a Hamiltonian path starting at vertex $u$, then it has also a $\mathrm{PC}(u,v)$. Assume that there is no Hamiltonian path starting at vertex $u$ and thus, $u$ is an articulation point by Theorem~\ref{thm:3K1u}. Denote the components of $G-\{u\}$ by $Q_1, Q_2$ in a way that $v \in Q_1$. Since $G$ is $3K_1$-free, $Q_1, Q_2$ are cliques. We can obtain one path starting at $v$ going through all vertices of $Q_1$ and the other path starting at $u$ and going through all vertices of $Q_2$.
\end{proof}

\subsection{$4K_1$-free graphs}\label{sec:4k1}

\begin{theorem}\label{thm:4k1}
Let $G$ be a connected $4K_1$-free graph. Then $G$ has a Hamiltonian path if and only if 
\begin{enumerate}[label=(\alph*)]
\item for every articulation point $x$, the graph $G-\{x\}$ has exactly 2 components, and 
\item there are no 3 articulation points inducing a triangle in $G$.
\end{enumerate}
Moreover, if $G$ does not have a Hamiltonian path, then $G$ has a path cover of size 2.
\end{theorem}

\begin{figure} 
\centering
\includegraphics[scale=0.8]{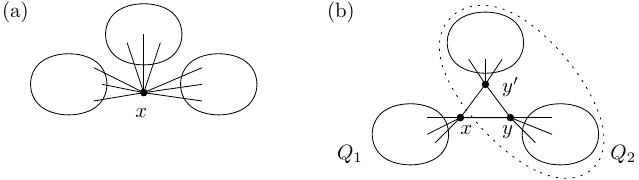}
\caption{An illustration of violation of the conditions in Theorem~\ref{thm:4k1}.}
\label{fig:4k1}
\end{figure}

\begin{proof}
If there exists an articulation point $x$ such that $G-\{x\}$ has at least 3 components, then $G$ cannot have a Hamiltonian path. Since $G$ is a $4K_1$-free graph, there cannot be more than 3 components and if there are exactly 3 components, then all of them must be cliques. Clearly, in this case, the path cover number of $G$ is equal to 2. 

If there are 3 articulation points $\{x,y,y'\}$ inducing a triangle in $G$, then  $G-\{x,y,y'\}$ has at most 3 components and all of them  are cliques (since $G$ is $4K_1$-free). At the same time each of the three articulation points must have a private component whose vertices are not adjacent to any of the remaining two articulation points.  Hence $G$ has the structure as depicted in Fig.~\ref{fig:4k1}~right, it clearly has no Hamiltonian path, but has a path cover of size two. See Figure~\ref{fig:4k1} as an illustration of these two cases.

Suppose that the conditions (a) and (b) are satisfied.

\emph{Case 1. $G$ is connected but not 2-connected.} 
Since condition (a) is satisfied, we have that for each articulation point $x$, $G-\{x\}$ has exactly 2 components. Let $x$ be an articulation point and $Q_1$, $Q_2$ be the components of $G-\{x\}$. Then at least one of $Q_1$ and $Q_2$, say $Q_1$, must be a clique and the other one, hence $Q_2$, must induce a $3K_1$-free graph.

Observe, that $G$ has a Hamiltonian path if and only if $G[Q_2]$ has a Hamiltonian path starting at some vertex $u \in Q_2$ adjacent to $x$,  and that happens if and only if $u$ is not an articulation point in $G[Q_2]$ (by Theorem~\ref{thm:3K1u}). 

Suppose that all neighbours of $x$ in $Q_2$ are articulation points of $G[Q_2]$. 
Let $y \in Q_2$ be such a vertex. Both components of $G[Q_2]-\{y\}$ are cliques, $y$ is adjacent to all vertices of at least one component and to at least one vertex in the other one, say $y'$. The vertices $y$ and $y'$ are the only possible articulation points in $G[Q_2]$, and thus $x$ is non-adjacent to all vertices of $Q_2$ except $y$ and possibly $y'$. If $x$ is not adjacent to $y'$, then $c(G-\{y\}) \geq 3$ and condition (a) would be violated. If $xy'$ is an edge, $y'$  is also an articulation point of $G[Q_2]$ and $y$ and $y'$ are the only neighbours of $x$ in $Q_2$. Then $x,y,y'$ are articulation points in $G$ inducing a triangle in $G$, violating condition (b). 

Thus, we can obtain a Hamiltonian path in $G$ by concatenating a Hamiltonian path in clique $G[Q_1]$, $x$ and a Hamiltonian path of $G[Q_2]$ starting in a neighbour $u \in {Q_2}$ of $x$. 

\emph{Case 2. $G$ is at least 2-connected.} In this case, $G$ has a Hamiltonian path by Proposition~\ref{prop:chvatal}.1.
\end{proof}

\begin{theorem}\label{thm:4k1_1con}
Let $G$ be a connected $4K_1$-free graph. The graph $G$ has a Hamiltonian path starting in a vertex $u$ if and only if the following conditions are satisfied.
\begin{enumerate}[label=(\alph*)]
\item The graph $G$  has a Hamiltonian path (i.e. the conditions of Theorem~\ref{thm:4k1} are satisfied).
\item The vertex $u$ is not an articulation point of $G$. 
\item Let $x\neq u$ be an articulation point of $G$ and let $Q_u$ be the component of $G-\{x\}$ containing $u$. If $G[Q_u]$ is  
not 2-connected, then for every articulation point $y\neq u$
of $G[Q_u]$, there exists a vertex $v \in {Q_u}$ 
adjacent to $x$ such that $u,v$ are in different components of $G[Q_u]-\{y\}$.
\item Let $x\neq u$ be an articulation point of $G$ and let $Q_u$ be the component of $G-\{x\}$ containing $u$. If $G[Q_u]$ is  
2-connected, then there exists $v \in {Q_u}$ adjacent to $x$ such that $v\neq u$ 
and $\{u,v\}$ does not form a minimum vertex cut in $G[Q_u]$.
\item There is no vertex $x$ such that $c(G-\{u,x\}) \geq 3$.
\item Let $G$ be 2-connected and let $\{x,y\}$ be a minimum vertex cut of $G$ such that $c(G-\{x,y\})=2$. Let $Q_1$ be a component of $G-\{x,y\}$ which is a clique and let $Q_2$ be the other component of $G-\{x,y\}$. If all neighbours of $x$ in ${Q_2}$ are articulation points of $G[Q_2]$, then $u \neq y$. 
\end{enumerate}
\end{theorem}

Note that if we allowed $y=u$ in condition (c) and (c) were violated with $y=u$, then condition (e) would be violated.

\begin{figure}
\centering
\includegraphics[width=\textwidth]{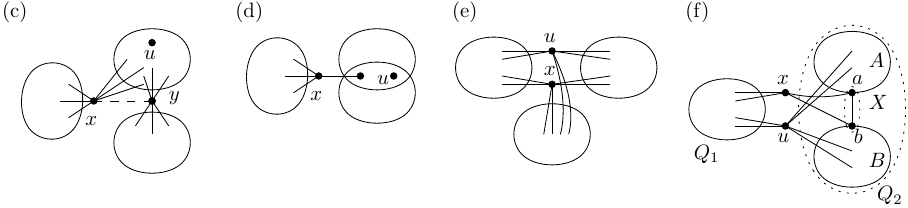}
\caption{An illustration of violation of conditions (c),(d),(e) and (f) in Theorem~\ref{thm:4k1_1con}. Dashed edge can be present or not.}
\label{fig:4k1_1con}
\end{figure}

\begin{proof} 
Clearly, if any of the conditions is not satisfied, then $G$ does not have a Hamiltonian path starting at vertex $u$. See Figure~\ref{fig:4k1_1con} for an illustration of possible situations.

Suppose that all of the conditions (a)--(f) are satisfied.

\textit{Case 1. $G$ is connected, but not $2$-connected.}

Let $x$ be an articulation point of $G$. Since condition (a) is satisfied, $G-\{x\}$ has two components. One of the components of $G-\{x\}$ is a clique, say $Q_1$, and the other one, say $Q_2$,  is a $3K_1$-free graph. 

\textit{Case 1a. $G[Q_2]$ is connected, but not 2-connected.} 
The vertex $u$ is not an articulation point of $G$ by condition (b), and thus $u\neq x$ and also $u$ is not the only neighbour of $x$ in its component of $G-\{x\}$ (unless this component contains only one vertex). Suppose first that $u\in Q_1$.  By condition (a), $G$ has a Hamiltonian path, say $P=P_1xP_2$ where $P_1$ is a Hamiltonian path in $G[Q_1]$ ending in a neighbour of $x$ and $P_2$ is a Hamiltonian path in $G[Q_2]$ starting in a neighbour of $x$. Since $G[Q_1]$ is a clique and $u$ is not the only neighbour of $x$ in $Q_1$ (unless $u$ being the only vertex of $Q_1$), $G[Q_1]$ contains a Hamiltonian path $P'_1$ starting in $u$ and ending in a neighbour of $x$. It follows that $P'=P'_1xP_2$ is a Hamiltonian path in $G$  starting in $u$. 

Suppose that $u \in Q_2$. The vertex $u$ cannot be an articulation point of $G[Q_2]$, otherwise $c(G-\{u,x\}) \geq 3$ and condition (e) of this theorem would be violated.
Observe that for every articulation point $y$ of $G[Q_2]$, $G[Q_2] - \{y\}$ has at most two components and they are cliques. By condition (c), $x$ has a neighbour, say $x'$, in the component of $G[Q_2] - \{y\}$ which does not contain $u$. If this $x'$ is an articulation point of $G[Q_2]$ (which happens when $x'$ is the only neighbour of $y$ in this component), then a repeated use of condition (c) shows that there is a neighbour, say $x''$, of $x$ in the component of $G[Q_2] - \{x'\}$ which does not contain $u$ such that $x''$ is not an articulation point of $G[Q_2]$, and thus this $x''$ lies in the component of $G[Q_2] - \{y\}$ which does not contain $u$.  Since $G[Q_2]$ is not 2-connected, $\{u,x'\}$ (nor $\{u,x''\}$ if $x'$ was an articulation point) cannot be a minimum vertex cut in $G[Q_2]$. 
Thus, since also the other two conditions of Theorem~\ref{thm:3K1_uv} are satisfied, there exists a Hamiltonian path of $G[Q_2]$ from $u$ to $x'$ (or $x''$).
Then we obtain a Hamiltonian path of $H$ by concatenating such a Hamiltonian path of $G[Q_2]$, $x$, and a Hamiltonian path of the clique $G[Q_1]$. 

\textit{Case 1b. $G[Q_2]$ is 2-connected, but not 3-connected.} 
If $u \in Q_1$, then a Hamiltonian path in $G$ can be easily obtained by concatenating a Hamiltonian path in $G[Q_1]$ starting at $u$, vertex $x$ and a Hamiltonian cycle of $G[Q_2]$ (it exists by Proposition~\ref{prop:chvatal}). 

If $u \in Q_2$, then $G$ has a Hamiltonian path if and only if $G[Q_2]$ has a Hamiltonian path starting in vertex $u$ and ending in a neighbour of $x$. By Theorem~\ref{thm:3K1_uv}, such a Hamiltonian path exists if there exists a neighbour $v \neq u$ of $x$ such that $\{u,v\}$ do not form a minimum vertex cut in $G[Q_2]$. But non-existence of such a neighbour would violate condition (d). 

\textit{Case 1c. $G[Q_2]$ is at least 3-connected.} In this case, $G[Q_2]$ is Hamiltonian connected which implies the existence of a Hamiltonian path in $G$. 

\textit{Case 2. $G$ is $2$-connected, but not $3$-connected.}

First suppose that there exists a (minimum) vertex cut $\{x,y\}$ such that $c(G- \{x,y \}) \geq 3$. Since $G$ is $4K_1$-free, $c(G-\{x,y\})>3$ cannot happen. Hence $c(G-\{x,y\}) = 3$ and 
all components of $G-\{x,y\}$ are cliques. Condition (e) implies that $u\neq x,y$. It is easy to see that $G$ has 
a Hamiltonian path starting in vertex $u$.

Assume from now on that for any vertex cut $\{x,y\}$ we have $c(G- \{x,y \}) =2$, and consider such a vertex cut $\{x,y\}$. Again, one of the components of $G- \{x,y \}$ is a clique, say $Q_1$, and the other one, say $Q_2$,  induces a $3K_1$-free subgraph of $G$. Denote by $X$ ($Y$) the neighbourhood of $x$ (of $y$, respectively) in $Q_2$. Observe that if there exists a Hamiltonian path of $G[Q_2]$ starting in a vertex from $X$ and ending in a vertex from $Y$, then $G$ has a Hamiltonian cycle, and thus it has a Hamiltonian path starting in any vertex. If $|X\cup Y|=1$, it must be $|Q_2|=1$ and $G$ contains a Hamiltonian cycle. Let further $|Q_2|>1$, and hence also $|X\cup Y|>1$. If there is no Hamiltonian path in $G[Q_2]$ starting in a vertex from $X$ and ending in a vertex from $Y$, then one of the conditions of Theorem~\ref{thm:3K1_uv} must be violated for any choice of different vertices, one from $X$ and the other one from $Y$. 

First suppose that $G[Q_2]$ is connected, but not $2$-connected. 
If $z$ is   an articulation point of $G[Q_2]$, then the components $G[Q_2] - \{z\}$ are cliques, $z$ is adjacent to all vertices in one component and to at least one vertex in the other component (if $z$ is adjacent to only one vertex in this component, then this vertex is also an articulation point and there are no other articulation points). Thus, $G[Q_2]$ contains at most 2 articulation points, and if it contains two of them, then they are adjacent. For a choice of two distinct vertices, one from $X$ and one from $Y$, the condition (c) of Theorem~\ref{thm:3K1_uv} for $G[Q_2]$ cannot be violated, since $G[Q_2]$ is not 2-connected. The conditions (a) or (b) could be violated in two following ways. 

\textit{Case 2a. $X$ or $Y$ contains only articulation points of $G[Q_2]$.} Since $G$ is $2$-connected,  only one of $X,Y$ can contain only articulation points, say $X$. From the observations above it follows, that $X$ contains at most two vertices and they are adjacent (if they are two). Denote the components of $G[Q_2] \setminus X $ as $A$ and $B$. Let $a \in X$ be a vertex adjacent to a vertex in $A$ and let $b \in X$ be a vertex adjacent to a vertex in $B$ (it is possible that $a=b$). Observe that $x$ must be adjacent to all vertices in $Q_1$. Otherwise, $x$, its non-neighbour in $Q_1$, a vertex from $A$ and a vertex from $B$ would form a $4K_1$.
The vertex $y$ must have at least one neighbour in $A$ and  at least one neighbour in $B$ such that they are not articulation points in $G[Q_2]$. 

If $u \in Q_1$, we obtain a Hamiltonian path starting in vertex $u$, going through all vertices of $Q_1$, $x$, $a$, all vertices in $A$ ending in a vertex adjacent to $y$, $y$ and all vertices in $B$. 
If $u=x$, then we obtain a Hamiltonian path starting in $x$ going through all vertices of $Q_1$, $y$ and a Hamiltonian path of $Q_2$ starting in a vertex in $Y$ which is not an articulation point of $G[Q_2]$. 
If $u=y$, then condition (f) would be violated. 

Suppose that $u \in Q_2$ and $u$ is not an articulation point of $G[Q_2]$. Since $y$ has neighbours in both $A$, $B$, we can always find a neighbour $y'$ of $y$ in $G[Q_2]$ such that $y'$ is not an articulation point of $G[Q_2]$, $y'$ and $u$ are in different components of $G[Q_2] \setminus X$. We obtain a Hamiltonian path starting in vertex $u$ by concatenating a Hamiltonian path of $G[Q_2]$ starting in $u$ and ending $y'$ (it exists by Theorem~\ref{thm:3K1_uv}) with a Hamiltonian path of $G[Q_1\cup\{x,y\}]$ starting in $y$ and ending in $x$. 

Finally, suppose that $u \in Q_2$ and $u$ is an articulation point in $G[Q_2]$. If $u$ is the only neighbour of $x$ in $Q_2$, then $\{ y,u \}$ form a vertex cut such that $c(G- \{y,u \}) =3$. If $u$ is not the only neighbour of $x$ in $Q_2$ (say $u \neq a$), then we obtain a Hamiltonian path starting in vertex $u$, going through all vertices of $B$, $y$, all vertices of $Q_1$, $x$, $a$ and all vertices of $A$.

\textit{Case 2b. There is an articulation point $z$ in $G[Q_2]$ such that both $X$ and $Y$ are  disjoint with one component of $G[Q_2]-\{z\}$.} This cannot happen, since $z$ would be an articulation point of entire $G$.

Let us summarize the case of $G[Q_2]$ not being 2-connected. Let $A$ and $B$ be the two components of that are created after the articulation points of $G[Q_2]$ are deleted. If Case~2a applies, we have shown that a Hamiltonian path starting in $u$ exists. If neither Case~2a nor Case~2b apply, then both $X$ and $Y$ intersect $A\cup B$, and $X\cup Y$ intersects both $A$ and $B$. Thus it is possible to choose $\overline{x}\in X\cap A$ and $\overline{y}\in Y\cap B$ (or $\overline{x}\in X\cap B$ and $\overline{y}\in Y\cap A$). For such a choice of $\overline{x},\overline{y}$, the conditions of Theorem~\ref{thm:3K1_uv} are satisfied, thus a Hamiltonian path in $G[Q_2]$ starting in $\overline{x}$ and ending in $\overline{y}$ exists, and a Hamiltonian path in $G$ starting in $u$ can be constructed as described above.  

Now suppose that $G[Q_2]$ is $2$-connected, but not $3$-connected.  In this case, conditions (a) and (b) for $G[Q_2]$ cannot be violated for any choice of two distinct vertices, but condition (c) could be violated. 

\textit{Case 2c. For any pair of distinct vertices $a \in X, b \in Y$, $\{a,b\}$ is a minimum vertex cut in $G[Q_2]$.} If $|X \cup Y| = 2$, then $G- \{X \cup Y\}$ has 3 components and this case has been handled already in the beginning of Case 2.

Suppose that $|X \cup Y| > 2.$ That means, that at least one of the $x,y$ has at least two different neighbours in $G[Q_2]$, say $x$. Take 3  different vertices $a,a',b \in X \cup Y$ such that $a,a'$ are neighbours of $x$ and $b$ is a neighbour of $y$. 

Let us recall that $G[Q_2]$ has a Hamiltonian cycle by Proposition~\ref{prop:chvatal}. We start with several observations about the structure of $G[Q_2]$. Consider any vertex $v \in X$. We have that $\{v,b\}$ is a minimum vertex cut in $G[Q_2]$ and thus $v$ is an articulation point in $G[Q_2-\{b\}]$. Graph $G[Q_2-\{b\}]$ remains $3K_1$-free, thus the components of $G[Q_2-\{v,b\}]$ are cliques and $v$ is adjacent to all vertices in one clique and to at least one vertex in the other clique. If $v$ has only one neighbour in the other clique, it is also an articulation point of $G[Q_2-\{b\}]$ and there are no other articulation points. This implies that $X=\{a, a'\}$. 
Similarly, if we consider vertex $b \in Y$, then $\{a,b\}$ is a minimum vertex cut in $G[Q_2]$ and $b$ must be an articulation point in $G[Q_2-\{a\}]$. Again, the components of $G[Q_2-\{a,b\}]$ are cliques, $b$ is adjacent to all vertices of at least one of them and to at least one vertex of the other clique, say $b'$. Thus, if $b'$ is the only neighbour of $b$ then $Y \subseteq \{b,b'\}$ and otherwise $Y = \{b\}$.  Several situations can occur, see Figure~\ref{fig:Q_2}. In each of them, it is easy to observe, that if $|X|=2$ then $G[Q_2 \cup \{x\}]$ has a Hamiltonian cycle. The same holds for $|Y|$ (or both at the same time). 

\begin{figure}
\centering
\includegraphics[scale=0.8]{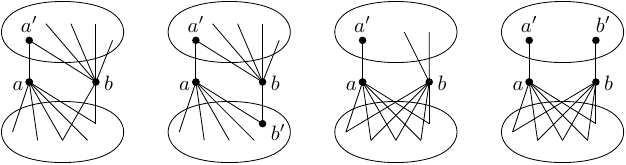}
\caption{An illustration of possible situations in $Q_2$ in Case~2c of Theorem~\ref{thm:4k1_1con}.}
\label{fig:Q_2}
\end{figure}

We distinguish several situations regarding the position of $u$ in $G$. 
\begin{itemize}
\item \textit{$u = x$ or $u=y$. } Let $u=x$, the situation for $u=y$ is symmetric. A Hamiltonian path starts at $u$, then goes through all vertices of $Q_1$, $y$ and then through all vertices of $Q_2$ ($G[Q_2]$ has a Hamiltonian cycle). 

\item \textit{$u \in Q_1$.} First suppose that $u$ is not the only neighbour of $y$ in $Q_1$ or $|Q_1|=1$. A Hamiltonian path starts at $u$, goes through all vertices of $Q_1$ ending in a neighbour of $y$ (in case $|Q_1|=1$ it goes directly to $y$), $y$ and then it goes along the Hamiltonian cycle of $G[Q_2 \cup \{x\}]$. 
Suppose that $u$ is the only neighbour of $y$ in $Q_1$ and $|Q_1|>1$. If $|Y| \geq 2$, i.e., if $Y=\{b,b'\}$, then a Hamiltonian path starts at $u$, goes through all vertices of $Q_1$ ending in a neighbour of $x$, $x$ and then it goes along the Hamiltonian cycle of $G[Q_2 \cup \{y\}]$.       
The case that $u$ is the only neighbour of $y$ in $Q_1$, $|Q_1|>1$ and $|Y| = 1$ cannot happen. Otherwise, we can obtain a $4K_1$ taking a vertex from $Q_1$ not adjacent to $y$, $y$, and one vertex from each component of $G[Q_2 - \{a,b\}]$. 

\item \textit{$u \in Q_2$.} Observe that if there exists a Hamiltonian path of $G[Q_2]$ starting at $u$ and ending at $X \cup Y$, then we are done. If no such path exists, then one of the conditions of Theorem~\ref{thm:3K1_uv} must be violated. Since $G[Q_2]$ is 2-connected, it must be condition (c). That means that for any $v \in X \cup Y$, $\{u,v\}$ form a minimum vertex cut in $G[Q_2]$. It is easy to observe (using Figure~\ref{fig:Q_2}) that there is no vertex $w \not \in X \cup Y$ such that $\{w,a\}$ and $\{w,b\}$ are both vertex cuts of $G[Q_2]$. Thus $u \in X \cup Y$. Fist suppose that $u = a$. We know that there is a Hamiltonian cycle of $G[Q_2]$ where $a,a'$ are consecutive (see Figure~\ref{fig:Q_2}). A Hamiltonian path of $G$ starts at $u$, goes through all vertices of $Q_2$ ending in $a'$, $x$, all vertices of $Q_1$ and $y$. We can proceed analogously if $u=a'$ or $u \in Y$ and $|Y|=2$. Thus, the last remaining case is $u \in Y \setminus X, |Y| = 1$ (thus $u = b$) and we show that this cannot happen. Since $G$ is $4K_1$-free, $y$ is adjacent to all vertices of $Q_1$. Observe that $\{x,u\}$ forms a minimum vertex cut in $G$ (since $G$ is 2-connected). One component of $G-\{x,u\}$ is $G[Q_1 \cup \{y\}]$ which is a clique and the other component is $G[Q_2 - \{u\}]$. Since $\{a,u\}$ and $\{a',u\}$ are cuts in $G[Q_2]$,  $a,a'$ are articulation points in $G[Q_2 - \{u\}]$. Thus $x$ is adjacent only to articulation points in $G[Q_2 - \{u\}]$ and condition (f) of this theorem would be violated. 
\end{itemize}

\textit{Case 3. $G$ is $3$-connected.} By Proposition~\ref{prop:chvatal}.2, $G$ has a Hamiltonian cycle, and hence it has a Hamiltonian path starting in an arbitrary vertex $u$. 
\end{proof}

\begin{theorem}\label{thm:4k1pc}
Let $G$ be a $4K_1$-free graph. If $G$ is connected, then $G$ has a path cover $\mathrm{PC}(u,v)$ if and only if the following conditions are satisfied.
\begin{enumerate}[label=(\alph*)]
\item For every articulation point $x$ such that $G-\{x\}$ has at least 3 components, $u,v$ are different from $x$ and belong to different components of $G-\{x\}$.
\item For any three articulation points $x,y,z$ inducing a triangle in $G$, 
\begin{itemize}
\item if both  $u,v$ are different from $x,y,z$, then $u,v$ are not in the same component of $G-\{x,y,z\}$,  
\item if exactly one of $u,v$ is in $\{x,y,z\}$, say $u$, then $v$ is not in the component of $G-\{x,y,z\}$ adjacent to $u$.
\end{itemize}
\item The pair $\{u,v\}$ does not form a vertex cut (not necessarily minimum) in $G$ such that $c(G-\{u,v\})\geq 3$.
\item Let $x\neq u$ be an articulation point of $G$ such that $c(G-\{x\}) \leq 2$, let $Q_1$ be the component of $G-\{x\}$ containing $u$, and let $Q_2$ be the other component  of $G-\{x\}$. Suppose that $G[Q_1]$ contains an articulation point $y\neq u$ and denote the component of $G[Q_1]-\{y\}$ containing $u$ by $J_1$ and the other component by $J_2$. If $v \in J_1$, $|J_1| > 2$, and $x$ is adjacent only to vertices from $J_1 \cup \{y\}$ in $G[Q_1]$, then there exists $x' \in J_1, x' \neq u,v $ adjacent to $x$ or $y$.
\end{enumerate}
If $G$ is not connected, then there is a $\mathrm{PC}(u,v)$ if and only if $G$ has exactly 2 components, $u,v$ are in different components and none of $u,v$ is an articulation point in its component. 
\end{theorem}

\begin{figure}
\centering
\includegraphics[width=\textwidth]{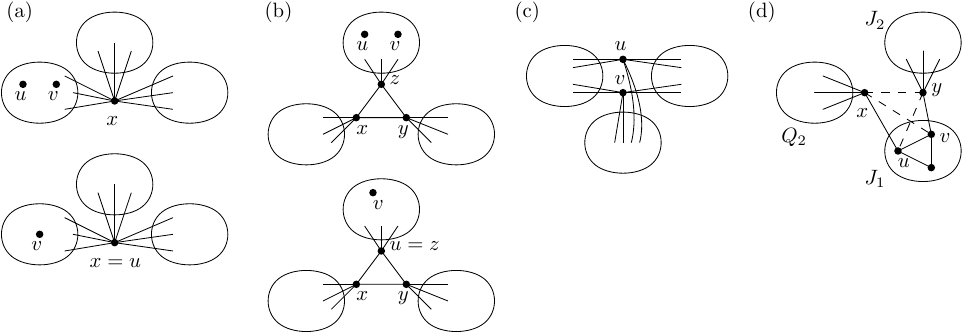}
\caption{An illustration of violations of conditions (a),(b),(c) and (d) in Theorem~\ref{thm:4k1pc}. Dashed edges can be present or not.}
\label{fig:4k1pathcover}
\end{figure}

\begin{proof}
If $G$ is not connected, the proof follows by Theorem~\ref{thm:3K1u}.

Suppose from now on that $G$ is connected. It is clear that if any of the conditions is not satisfied, then $\mathrm{PC}(u,v)$ does not exist. See Figure~\ref{fig:4k1pathcover} for an illustration of possible violations of the conditions. 

Suppose that all of the conditions (a)--(d) are satisfied. If there exists a Hamiltonian path starting in $u$ (without loss of generality, we can interchange $u$ and $v$), then there is also a path cover $\mathrm{PC}(u,v)$. 

If such a path does not exist, then at least one of the conditions in Theorem \ref{thm:4k1_1con} is violated.  

\textit{Case 1: $G$ does not have a Hamiltonian path.} That means, that condition (a) of Theorem \ref{thm:4k1_1con} is violated. If $G$ does not have any Hamiltonian path, then by Theorem \ref{thm:4k1}, there exists an articulation point $x$, such that $G-\{x\}$ has at least 3 components or a triangle induced by articulation points $x,y,z$ in $G$ such that $c(G-\{x,y,z\}) \geq 3$. In both cases, since $G$ is $4K_1$-free,  there are exactly 3 components  and all of them are cliques. 

\textit{Case 1a: There exists an articulation point $x$, such that $c(G-\{x\})  \geq 3$. } By condition (a) of this theorem we have that $x \neq u,v$ and $u,v$ are in different components of $G-\{x\}$. Denote the component containing $u$ by $Q_u$, the component containing $v$ by $Q_v$, and the third component by $Q_w$. If $Q_u=\{u\}$, we obtain the first path of $\mathrm{PC}(u,v)$ starting at $u$,  continuing to $x$ and then to all vertices of $Q_w$. The second path starts at $v$ and traverses all vertices of $Q_v$. Similarly if $Q_v$ has only one vertex. For the rest of this case, suppose that $|Q_u|>1$ and $|Q_v|>1$. If $u$ is the only neighbour of $x$ in $Q_u$ and $v$ is the only neighbour of $x$ in $Q_v$, then $\{u,v\}$ is a vertex cut in $G$ with $c(G-\{u,v\})\geq 3$. That would violate the condition (c) of this theorem. Thus, $x$ has at least two neighbours in at least one of $Q_u$, $Q_v$, say in $Q_u$. Then we obtain the first path of $\mathrm{PC}(u,v)$ starting at $u$, traversing all the vertices of $Q_u$ (ending in a neighbour of $x$), then continuing to $x$ and then to all vertices of $Q_w$. The second path starts at $v$ and traverses all vertices of $Q_v$. 

\textit{Case 1b: There exists a triangle induced by articulation points $x,y,z$ in $G$.} 
Denote the component adjacent to $x$ by $Q_x$, the component adjacent to $y$ by  $Q_y$, and the component adjacent to $z$ by $Q_z$.  
Observe that $x$ is adjacent to all vertices in $Q_x$, otherwise $x$, a vertex from $Q_x$ not adjacent to $x$, a vertex from $Q_y$, and a vertex from $Q_z$ would create a $4K_1$. Analogously, $y$ is adjacent to all vertices in $Q_y$ and $z$ is adjacent to all vertices in $Q_z$.  

If both $u,v$ belong to $\{x,y,z\}$, then they would form a vertex cut such that $c(G-\{u,v\}) \geq 3$ and the condition (c) of this theorem would be violated. 

If exactly one of $u,v$ belongs to $\{x,y,z\}$, say $u=z$,  then we can obtain one path starting at $u$ traversing all vertices of $Q_z$ and the second path starting at $v$ and traversing all vertices of cliques $Q_x, Q_y$. 

If none of $u,v$ belongs to $\{x,y,z\}$ and they are in different components of $G-\{x,y,z\}$, then it is clear that $G$ has a $\mathrm{PC}(u,v)$.

The last case of both $u,v$ belonging to the same component of $G-\{x,y,z\}$ is forbidden by condition (b) of this theorem.

\textit{Case 2. The graph $G$ does have a Hamiltonian path, but it has no Hamiltonian path starting in vertex $u$ nor in $v$}. That means that condition (a) of Theorem~\ref{thm:4k1_1con} is satisfied, but both for $u$ and for $v$, at least one of the conditions (b)--(f) of Theorem~\ref{thm:4k1_1con} is violated. We will show by a case analysis that even when one of these conditions is violated for just one of these vertices, say $u$, we can always find a path cover $\mathrm{PC}(u,v)$. 

\textit{Case 2a. Condition (b) of Theorem~\ref{thm:4k1_1con} is violated for $u$.} 
That means that $u$ is an articulation point of $G$.  One of the components of $G-\{u\}$ is a clique, say $Q_1$, and the other one, say $Q_2$, induces a $3K_1$-free graph (note here that since condition (a) of Theorem~\ref{thm:4k1_1con} is satisfied, $c(G-\{u\})\le 2$). 

First suppose that $v \in Q_1$. Let $P=P_1uP_2$ be a Hamiltonian path in $G$, with $P_1$ being a Hamiltonian path in $G[Q_1]$ ending in a neighbour of $u$ and $P_2$ being a Hamiltonian path in $G[Q_2]$ starting in a neighbour of $u$. Then $uP_2$ can be taken as one path and any Hamiltonian path of $G[Q_1]$ starting in $v$ as the other one for a path cover $\mathrm{PC}(u,v)$. 

Now suppose that $v \in Q_2$. 
If $v$ is an articulation point in $G[Q_2]$, then $G-\{u,v\}$ has at least 3 components, which is not possible (condition (c) of this theorem would be violated). If $v$ is not an articulation point of $G[Q_2]$, then there exists a Hamiltonian path of $G[Q_2]$ starting in $v$, which together with any Hamiltonian path of $G[Q_1\cup\{u\}]$ starting in $u$ forms a path cover $\mathrm{PC}(u,v)$.

\textit{Case 2b. The conditions (a)--(b) of  Theorem~\ref{thm:4k1_1con} are satisfied but condition (c) is violated for $u$.} 
Namely, $u$ is not an articulation point of $G$. Violation of condition (c) of Theorem~\ref{thm:4k1_1con} means that we have the following situation. There are vertices $x\neq u$, $y\neq u$ such that
\begin{itemize}
\item $x$ is an articulation point of $G$, 
\item $G-\{x\}$ has two components $Q_1$ and $Q_2$,
\item $Q_1$ contains $u$,
\item $y$ is an articulation point of $G[Q_1]$,
\item $G[Q_1]-\{y\}$ has two components $J_1$ and $J_2$,
\item if $J_1$ is the component containing $u$, then $x$ is non-adjacent to all vertices from $J_2$.
\end{itemize} 

Note that $G[Q_2], G[J_1], G[J_2]$ are cliques and $x$ is adjacent to all vertices in at least one of the sets $Q_2$ and $J_1$ (otherwise we would have a $4K_1$). 
We will distinguish several cases depending on the position of $v$. 
\begin{itemize}
\item 
Suppose that $v \in Q_2$ and there exists $x' \in Q_2$ adjacent to $x$, $x'\neq v$. If $u$ is not an articulation point of $G[Q_1]$ then we can obtain one path of $\mathrm{PC}(u,v)$ starting in $v$, traversing all vertices of $G[Q_2]$ (ending in $x'$), $x$. The second path is a Hamiltonian path of $G[Q_1]$ starting at $u$ (which existence is ensured by Theorem~\ref{thm:3K1u}). 
If $u$ is an articulation point of $G[Q_1]$, then it is adjacent to $y$ and also to all vertices in $J_1$. If $x$ is adjacent only to $u$ in $G[Q_1]$, then $c(G-\{u\}) \geq 3$ and condition (a) of this theorem is violated. Thus, $x$ is adjacent to $y$ or a vertex of $J_1$ other than $u$, say $u'$. If $x$ is adjacent to $y$ then we obtain one path of $\mathrm{PC}(u,v)$ starting in $v$, traversing all vertices of $G[Q_2]$ (ending in $x'$), $x$, $y$ and then all vertices of $G[J_2]$. The second path is a Hamiltonian path of $G[J_1]$ starting at $u$. If $x$ is adjacent to $u'$ then we obtain one path of $\mathrm{PC}(u,v)$ starting in $v$, traversing all vertices of $G[Q_2]$ (ending in $x'$), $x$, $u'$ and then all vertices of $J_1 \setminus \{u\}$. The second path starts at $u$ and then goes through $y$ and all vertices of $G[J_2]$. 
\item
Suppose that $v \in Q_2$ and $v$ is the only neighbour of $x$ in $G[Q_2]$. If $|Q_2| = 1$, then we can proceed as in the previous case. If $|Q_2| > 1$, then $x$ must be adjacent to all vertices in $J_1$, otherwise we would obtain a $4K_1$. 
\begin{itemize}
\item If $|J_1| = 1$ (it contains only vertex $u$) and $xy$ is not an edge, then $\{u,v\}$ forms a vertex cut in $G$ such that $c(G-\{u,v\})\geq 3$, which would violate condition (c) of this theorem. 
\item If $|J_1| = 1$ and $xy$ is an edge, then we can obtain one path starting at $v$ and traversing all vertices of $G[Q_2]$ and the second path traversing $u$, $x$, $y$, and all vertices in $G[J_2]$. 
\item Suppose $|J_1| > 1$. If $u$ is not an articulation point of $G[Q_1]$, take a Hamiltonian path $P_1$ of $G[Q_1]$ starting at $u$ and denote the first edge by $uu'$. Observe that $u' \in J_1$ and thus $xu'$ is an edge. Create $P'_1$ from $P_1$ by replacing the $uu'$ by the path $ uxu'$. Take $P'_1$ as one path of $\mathrm{PC}(u,v)$ and as the second path, take a path starting in $v$ and traversing all vertices of $Q_2$. 
If $u$ is an articulation point of $G[Q_1]$, then it is adjacent to $y$ and also all vertices in $G[J_1]$. If $xy$ is an edge, one path starts at $u$, then goes through all vertices of $G[J_1]$, $x$, $y$ and all vertices of $G[J_2]$. The second path is a path starting in $v$ traversing all vertices of $G[Q_2]$. If $xy$ is not an edge, then $c(G-\{u,v\}) \geq 3$ and condition (c) of this theorem is violated.
\end{itemize}
\item
If $v = x$ and $u$ is not an articulation point of $G[Q_1]$, then we obtain one path of $\mathrm{PC}(u,v)$ starting in $v$ and traversing all vertices of $G[Q_2]$, and  the second path is a Hamiltonian path of $G[Q_1]$ starting at $u$. 
If $u$ is an articulation point of $G[Q_1]$, then $c(G-\{u,v\}) \geq 3$ and condition (c) of this theorem is violated.
\item
If $v \in J_1$, then each Hamiltonian path of $G$ must traverse its vertices in this order (or in the reverse one): vertices of $Q_2$, $x$, vertices of $J_1$, $y$, vertices of $J_2$. Recall that $G[J_1]$ induces a clique. Our goal is to show that there exists a Hamiltonian path of $G$ which traverses vertices of $G$ in the following order: vertices of $Q_2$, $x$, some vertices of $J_1$, $u$, $v$ (or $v$,$u$), the rest of vertices of $J_1$, $y$, vertices of $J_2$. Obviously, this would give us a path cover $\mathrm{PC}(u,v)$. If $|J_1|=2$, then we are done. Thus, suppose $|J_1| > 2$. Several cases can occur.
\begin{itemize}
\item 
If $x$ and $y$ have exactly one common neighbour and no other neighbours in $J_1$, then $G$ has no Hamiltonian path and condition (a) of Theorem~\ref{thm:4k1_1con} is violated. 
\item 
Suppose that $x$ is adjacent only to $u$ (and possibly also $v$) in $J_1$ and $y$ is adjacent only to $v$ (and possibly also $u$) in $J_1$ (or vice versa). If $xy$ is not an edge, then the condition (c) of this theorem would be violated. If $xy$ is an edge, then the condition (d) of this theorem would be violated. 
\item 
Without loss of generality (the situation for $y$ is symmetric), suppose that $x$  has a neighbour $x'$ in $J_1$ other than $u,v$, and $y'$ is a neighbour of $y$ other than $x'$ (possibly $y'=u$ or $y'=v$). We can traverse vertices in the order: vertices of $Q_2$, $x$, $x'$, vertices of $J_1 \setminus \{x',y',u,v\}$ in an arbitrary order, $u$, $v$ (or vice versa), and finally $y'$ to get the sought Hamiltonian path. Note that it is possible for $y'$ to coincide with $u$ or $v$.  
\end{itemize}
\item  
If $v = y$, then we can proceed as in the case of $v=x$ since the situation is symmetric (just the roles of $x, Q_2$ and $y, J_2$ are interchanged).
\item 
If $v \in J_2$, then we can proceed as in the case of $v \in Q_2$ since the situation is symmetric (just the roles of $x, Q_2$ and $y, J_2$ are interchanged).
\end{itemize}

\textit{Case 2c. The conditions (a)--(c) of Theorem~\ref{thm:4k1_1con} are satisfied, but condition (d) is violated for $u$.}
Let $x\neq u$ be an articulation point of $G$ causing this violation. Let $Q_1$ be the component of $G-\{x\}$ containing $u$  and let $Q_2$ be the other component. Suppose that $G[Q_1]$ is 2-connected and that for every $x' \in {Q_1}$ adjacent to $x$, $x'\neq u$, the pair $\{u,x'\}$ forms a minimum vertex cut in $G[Q_1]$. 

Observe that if $u$ is adjacent to $x$, then $u$ cannot be the only neighbour of $x$ in $Q_1$, since in such a case $u$ would be an articulation point of $G$ and condition (b) of Theorem~\ref{thm:4k1_1con} would be violated. Hence $x$ has a neighbour, say $z$, in $Q_1$ such that $z\neq u$. Then $\{z,u\}$ is a vertex cut in $G[Q_1]$ and therefore $G[Q_1]-\{u\}$ is not 2-connected.  Moreover, $z$ is an articulation point of $G[Q_1]-\{u\}$ for every neighbour $z$ of $x$ in $Q_1$ such that $z\neq u$.

Suppose that $y\neq u$ is a neighbour of $x$ in $Q_1$. Denote the components of $G[Q_1]-\{u,y\}$ by $J_1$, $J_2$. Since $G[Q_1]-\{u\}$ is $3K_1$-free, both $J_1$, $J_2$ are cliques and $y$ is adjacent to all vertices in one component of $G[Q_1]-\{u,y\}$, say $J_1$. Let $y'$ be a neighbour of $y$ in $J_2$. If $y'$ is the only neighbour of $y$ in $J_2$, then $y'$ is also an articulation point of $G[Q_1]-\{u\}$ and there are no other articulation points. Thus, apart from $y$, $x$ can be adjacent only to $y'$  in $G[Q_1]-\{u\}$. If $y$ has more neighbours in $J_2$, then $y$ is the only neighbour of $x$ in  $G[Q_1]-\{u\}$.

Since $G$ is a $4K_1$-free graph, $x$ is adjacent to all vertices of $G[Q_2]$. Recall that $G[Q_1]$ is a 2-connected $3K_1$-free graph and thus, it has a Hamiltonian cycle. Moreover,  each of the vertices $u$ and $y$ has at least one neighbour in $J_1$ and at least one neighbour in $J_2$, and if $|J_2|>1$, then $u$ and $y$ together have at least two different neighbours in $J_2$.  Several cases can occur depending on the position of $v$. 
\begin{itemize}
\item 
If $v \in Q_2 \cup \{x\}$, then since $G[Q_2 \cup \{x\}]$ is a clique and $G[Q_1]$ has a Hamiltonian cycle, we are done. 
\item 
If $v \in J_1$, then one path starts at $v$ and traverses all vertices of $J_1$, while the other one starts at $u$, traverses all the vertices of $J_2$, then $y$, $x$ and all the vertices of $Q_2$. 
\item 
If $v \in J_2$,  then one path starts at $v$ and traverses all vertices of $G[J_2]$, while the other one starts at $u$ and traverses all vertices of $G[J_2]$, $y$, $x$, and then all the vertices of $G[Q_2]$.
\item 
If $v = y$ and $xy' \in E(G)$, $y'$ is the only neighbour of $y$ in $J_2$, and if $|J_2|>1$, $u$ has a neighbour in $J_2$ different from $y'$. Then as one path we take a path starting at $u$, traversing all the vertices of $J_2$ to $y'$, continuing to $x$ and finally traversing all the vertices of $Q_2$. The other path starts at $u$ and traverses all the vertices of $J_1$. 
\item
If $v = y$ and $xy' \not \in E(G)$, then $\{u,v\}$ violates the condition (c) of this theorem.
\end{itemize}

\textit{Case 2d. The conditions (a)--(d) of Theorem~\ref{thm:4k1_1con} are satisfied, but condition (e) is violated for $u$.}  There exists a vertex $x$ such that $c(G-\{u,x\}) \geq 3$. Obviously $u\neq x$ (since $G$ would have no Hamiltonian path otherwise) and $v\neq x$ (since condition (c) of this theorem would be violated otherwise). 

Since $G$ is $4K_1$-free, $G-\{u,x\}$ has exactly 3 components. Let us denote them by $Q_1, Q_2, Q_3$. All three of these components must be cliques of $G$. Let $P$ be a Hamiltonian path in $G$ (which is assumed to exist). The path must first visit all vertices of one component, say $Q_1$, then vertex $x$, then the vertices of another component, say $Q_2$, then vertex $u$, and then all vertices of the last component $Q_3$. Let $P_i$ be the part of $P$ in $Q_i$, for $i=1,2,3.$ We consider four cases:

\begin{itemize}
\item Suppose $v\in Q_1$. If $|Q_1|=1$, then we take as the first path the path starting in $v$, followed by $x$ and $P_2$, and as the other one the path starting in $u$ and followed by $P_3$. If $|Q_1|>1$ and $x$ has a neighbour $x_1\neq v$  in $Q_1$, then we take as the first path the path starting in $v$, traversing all vertices of $Q_1$ to $x_1$, followed by $x$ and $P_2$, and as the other one the path starting in $u$ and followed by $P_3$. Hence suppose that $|Q_1|>1$ and $v$ is the only neigbour of $x$ in $Q_1$. Condition (c) of this theorem implies that $G[\{x\}\cup Q_2 \cup Q_3]$ is connected, and thus $x$ is adjacent to a vertex $x_3\in Q_3$. If $|Q_3|=1$, then as one path we take the path starting in $u$, followed by $x_3$, $x$ and $P_2$, and as the other one a path starting in $v$ and followed by $P_1^{-1}$. If $|Q_3|>1$, then (since $G$ is 2-connected) $u$ and $x$ have at least two neighbours in $Q_3$, and we may choose $x_3$ to be different from a neighbour $u_3$ of $u$ in $Q_3$. Then as the first path we take a path starting in $u$, followed by a Hamiltonian path in $G[Q_3]$ from $u_3$ to $x_3$, followed by $x$ and $P_2$,  and as the other one a path starting in $v$ and followed by $P_1^{-1}$. 
\item We have already observed that $v\neq x$. 
\item If $v\in Q_2$ and $|Q_2|=1$, we take as one path the path starting in $v$, continuing with $x$ followed by $P_1$ traversed in the reverse order. The other path will start in $u$ and continue as $P_3$.
\item Suppose $v\in Q_2$ and $|Q_2|>1$. If $x$ has a neighbour $x_2\neq v$ in $Q_2$, we take as the first path the path starting in $v$, followed by a Hamiltonian path in $G[Q_2]$ ending in $x_2$, followed by $x$ and $P_1^{-1}$, and as the other one the path starting in $u$ and followed by $P_3$. If $v$ is the only neighbour of $x$ in $Q_2$, then (since $G$ is 2-connected) $x$ has a neighbour in $Q_3$. Similarly as in the case of $v\in Q_1$, we take as one path a path starting in $u$, followed by a Hamiltonian path in $G[Q_3]$ ending in a neighbour of $x$, followed by $x$ and $P_1^{-1}$, and as the other one the path $P_2$ (which necessarily starts in $v$). 
\item If $v\in Q_3$, the first path will start in $u$, followed by $P_2$ traversed in the reversed order, then by $x$ and finally by $P_1$, again traversed in the reversed order. The second path will start in $v$ and traverse all vertices of $Q_3$.  
\end{itemize}

\textit{Case 2e. The conditions (a)--(e) of Theorem~\ref{thm:4k1_1con} are satisfied, but condition (f) is violated for $u$.} Let $x$ and $y=u$ be two different vertices causing this violation. Let $Q_1$ be a component of $G-\{x,y\}$ inducing a clique and let $Q_2$ be the other component. Note that $G[Q_2]$ is not 2-connected, because the violation of condition (f) implies that $x$ is adjacent only to articulation points of $G[Q_2]$ in $Q_2$, and $x$ is adjacent to at least one vertex of $Q_2$ (otherwise $\{x,y\}$ would not be a minimum cut).  

We have just observed that $Q_2$ contains at least one pair of non-adjacent vertices which are both non-adjacent to $x$. Therefore,  vertex $x$ is adjacent to all the vertices in $Q_1$, otherwise $G$ would not be $4K_1$-free. 

The graph induced by $G[Q_2]$ is $3K_1$-free. As we have already seen several times, a connected but not 2-connected $3K_1$-free graph has a very special structure. Let $w$ be  an articulation point of $G[Q_2]$ adjacent to $x$ and let the components of $G[Q_2]-\{w\}$ be $J_1$, $J_2$. Then both $J_1$ and $J_2$ induce cliques, $w$ is adjacent to all vertices of one of them, say $J_1$, and to at least one vertex, say $w'$, in $J_2$. Then one of the three cases happens:
\begin{enumerate}
\item $|J_2|=1$,
\item $|J_2|>1$ and $w$ has at least two neighbours in $J_2$,
\item $|J_2|>1$ and $w'$ is the only neighbour of $w$ in $J_2$.
\end{enumerate} 
In the first two cases, $w$ is the only articulation point in $G[Q_2]$, while in the last case, $w'$ is also an articulation point and $G[Q_2]$ has exactly two articulation points ($w$ and $w'$). Thus the neighbours of $x$ in $Q_2$ are $w$ and possibly $w'$ (but only in the case 3. above). 

Since $G$ is 2-connected, the vertex $u$ has a neighbour in $J_1$ and also a neighbour in $J_2$ which is not  an articulation point of $G[Q_2]$. Several cases can occur depending on position of $v$. 
\begin{itemize}
\item 
If $v \in Q_1$, then we  obtain one path starting at vertex $v$, traversing all the vertices of $G[Q_1]$, and ending in $x$. We obtain the second path by concatenating $u$ and a Hamiltonian path of $G[Q_2]$ starting in a neighbour of $u$ which is not an articulation point of $G[Q_2]$.
\item If $v=x$, we obtain the first path by concatenating $v$ and a Hamiltonian path in $G[Q_1]$, and the second path exactly as in the previous case. 
\item 
If $v \in Q_2$ and $v$ is not an articulation point of $G[Q_2]$, then we obtain one path starting at vertex $u$, traversing all the vertices of $G[Q_1]$, and $x$. The second path is a Hamiltonian path of $G[Q_2]$ starting in $v$. 
\item 
If $v \in Q_2$, $v$ is an articulation point of $G[Q_2]$ and $x' \neq v$ is a neighbour of $x$ in $G[Q_2]$, then $G$ actually has a Hamiltonian path starting in vertex $v$. In this case the vertex $x'$ is also an articulation point of $G[Q_2]$, and thus the case 3. above applies, and $\{v,x'\}=\{w,w'\}$. Since the situation is symmetric, we may assume without loss of generality that $v=w$ and $x'=w'$. A Hamiltonian path starting in vertex $v$ is then obtained by starting in $v$, traversing all vertices of $J_1$ to a neighbour of $u$, continuing to $u$, traversing all vertices of $Q_1$, continuing to $x$, then to $x'$ and finishing by traversing the remaining vertices of $J_2$. The existence of a path cover $\mathrm{PC}(v,u)$ then follows immediately.
\item 
If $v \in Q_2$ and $v$ is an articulation point of $G[Q_2]$, it cannot be the only neighbour of $x$ in $G[Q_2]$, since $\{u,v\}$ would violate the condition (c) of this theorem otherwise.
\end{itemize}
This concludes the proof. 
\end{proof}

\subsection{$5K_1$-free graphs}\label{sec:5k1}

In this section we present necessary and sufficient conditions for a connected $5K_1$-free graph to contain a Hamiltonian path. 

\begin{theorem}\label{thm:5k1}
Let $G$ be a connected $5K_1$-free graph. The graph $G$ has a Hamiltonian path if and only if the following conditions are satisfied. 
\begin{enumerate}[label=(\alph*)]
\item For every articulation point $x$, $c(G-\{x\}) \leq 2$.
\item There are no 3 articulation points inducing a triangle in $G$.
\item If $G$ is not $2$-connected, and $x$ is an articulation point of $G$ such that one of the components of $G-\{x\}$, denoted by $Q_1$, is a clique and the other one, denoted by $Q_2$, induces a $4K_1$-free graph, then there exists a vertex $u\in Q_2$ adjacent to $x$ such that $G[Q_2]$ has a Hamiltonian path starting in $u$. 
\item If $G$ is $2$-connected, but not $3$-connected, then for every minimum vertex cut $\{x,y\}$ in $G$,  $c(G-\{x,y\}) \leq 3$. 
\end{enumerate}
\end{theorem}

\begin{figure}
\centering
\includegraphics[width=\textwidth]{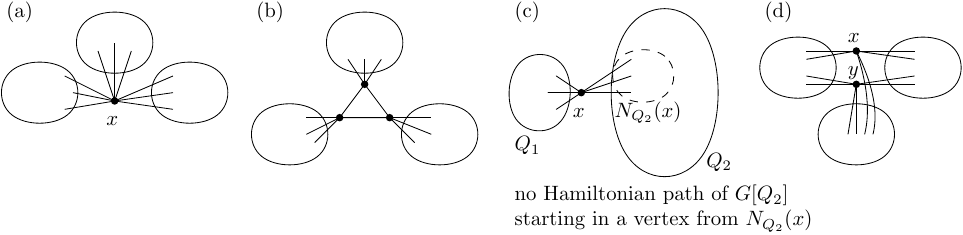}
\caption{An illustration of violations of the conditions (a),(b),(c) and (d) in Theorem~\ref{thm:5k1}.}
\label{fig:5k1}
\end{figure}

\begin{proof} We will use notation $N_Q(x)=\{y:y\in Q \wedge xy\in E(G)\}$ for a vertex $x\in V(G)$ and a set $Q$ of vertices of $G$.

It is easy to observe that if any of the conditions is violated, then $G$ does not have a Hamiltonian path. See Figure~\ref{fig:5k1} for an illustration of violations of these conditions. Suppose that all of the conditions (a)--(d) are satisfied. We will show that then $G$ has a Hamiltonian path. 
  
\textbf{Case 1.} \textit{Suppose that $G$ is connected, but not $2$-connected.} 

Let $x$ be an articulation point of $G$. By condition (a), $G-\{x\}$ has exactly two components of connectivity. First suppose that one of the components of $G-\{x\}$ (say $Q_1$) is a clique  and the other one (say $Q_2$) is $4K_1$-free. Clearly, $G$ has a Hamiltonian path if and only if there exists  $u\in N_{Q_2}(x)$  such that there exists a Hamiltonian path of $G[Q_2]$ starting in $u$, and this is guaranteed by condition (c) of this theorem.

If both $Q_1$, $Q_2$ induce $3K_1$-free graphs, then $G$ has a Hamiltonian path if and only if $x$ has a neighbour  $ x_1 \in Q_1$ and a neighbour $x_2 \in Q_2$ such that $x_1$ is not an articulation point in $G[Q_1]$ and $x_2$ is not an articulation point in $G[Q_2]$ (because, by Theorem~\ref{thm:3K1u}, in these, and only in these cases $G[Q_i]$ has a Hamiltonian path, for $i=1,2$). Recall that a connected $3K_1$-free graph has at most two articulation points, and if it has two of them, then they are adjacent. If $x$ is adjacent only to one articulation point in $G[Q_1]$ (or in $G[Q_2]$) and no other vertices of this component, then condition (a) of this theorem is violated. If $x$ is adjacent to two articulation points in $G[Q_1]$ (or in $G[Q_2]$) and no other vertices of this component, then condition (b) is violated.

\textbf{Case 2.} \textit{ Suppose that $G$ is $2$-connected, but not $3$-connected.}

Condition (d) of this theorem implies that for every vertex cut $\{x,y\}$ of size 2 in $G$, $c(G-\{x,y\}) \leq 3$.
If there is a vertex cut $\{x,y\}$ of size 2 in $G$ such that $c(G-\{x,y\}) = 3$, then two components of $G-\{x,y\}$, say $Q_1,Q_2$, are cliques  and the third component, say $Q_3$, induces a $3K_1$-free graph. In this case, $G$ has a Hamiltonian path if and only if there is a vertex in ${Q_3}$ which is not an articulation point of $G[Q_3]$ and which is adjacent to $x$  or $y$. Such a vertex always exists since $G$ is $2$-connected.

Suppose from now on that for every vertex cut $\{x,y\}$ of size 2 in $G$, $c(G-\{x,y\})= 2$. Fix a vertex cut $\{x,y\}$ in $G$ and let $Q_1,Q_2$ be the two components of $G-\{x,y\}$.

\textbf{Case 2a.} \textit{ Both $Q_1,Q_2$ induce $3K_1$-free graphs.}  

Recall that $3K_1$-free graphs have path covers of size two starting at any pair of vertices by Theorem \ref{thm:3K1pathcover}. 

If at least one of $G[Q_1]$ and $G[Q_2]$ has a Hamiltonian path  starting in a neighbour of $x$ and ending in a neighbour of $y$, then $G$ has a Hamiltonian path. If this is not the case, then for each $i=1,2$ and for every choice of $x'\in N_{Q_i}(x)$, $y'\in N_{Q_i}(y)$, $x'\neq y'$, at least one  of the conditions (a)--(c) of Theorem \ref{thm:3K1_uv} is violated. We will show that even if this is the case, $G$ contains a Hamiltonian path anyway.

\medskip\noindent 
{\em Claim 7.1} \ If for some $i=1,2$, all neighbours of at least one of the vertices $x,y$ are articulation points of $G[Q_i]$, then $G$ has a Hamiltonian path. 
(Note that in this case all choices of $(x',y')$ such that $x'\in N_{Q_i}(x)$, $y'\in N_{Q_i}(y)$ violate condition (a) of Theorem~\ref{thm:3K1_uv}.)

\medskip\noindent
{\em Proof of the claim.} Suppose without loss of generality that $N_{Q_1}(x)\subseteq Art(G[Q_1])$. Since $x$ has at least one neighbour in $Q_1$, this means that $G[Q_1]$ is not 2-connected, and as we have already seen several times, $|Art(G[Q_1])|\le 2$.  Let $u$ be an articulation point of $G[Q_1]$ adjacent to $x$. Denote the connected components of $G[Q_1 - \{u\}]$ by $J_1, J_2$. Since $G[Q_1]$ is $3K_1$-free, $J_1$ and $J_2$ are cliques, $u$ is adjacent to all vertices of one component, say $J_1$, and $u$ has at least one neighbour $v \in J_2$. If $v$ is the only neighbour of $u$ in $J_2$ and $|J_2|>1$, then $v$ is also an articulation point of $G[Q_1]$ and there are no other articulation points. Thus, $N_{Q_1}(x) \subseteq \{u,v\}$ if $v$ is an articulation point of $G[Q_1]$, and $N_{Q_1}(x) = \{u\}$ otherwise. Since $G$ is 2-connected, $y$ has neighbours both in $J_1$ and $J_2$. 

It follows that $G[Q_1 \cup \{x,y\}]$ has a Hamiltonian path starting in a vertex of $J_2$, going through all vertices of $G[J_2]$ to  a neighbour of $y$, proceeding to $y$, then traversing  all vertices of $G[J_1]$, proceeding to $u$ and ending in $x$. If $G[Q_2]$ has a Hamiltonian path starting at a neighbour of $x$, then we are done. 

Suppose for a contradiction that $G[Q_2]$ does not have a Hamiltonian path starting at a neighbour of $x$. By Theorem~\ref{thm:3K1u}, we have that $N_{Q_2}(x)\subseteq Art(G[Q_2])$. Then $x$ together with one vertex from each component of $G - (\{x,y\} \cup N_{Q_1}(x) \cup N_{Q_2}(x))$ create a copy of $5K_1$ in $G$, a contradiction. \qed

\medskip\noindent
{\em Claim 7.2} \ Suppose that for some $i=1,2$, $G[Q_i]$ is not 2-connected and that neither $N_{Q_i}(x)$ nor $N_{Q_i}(y)$ is a subset of $Art(G[Q_i])$. Then $x$ and $y$ have neighbours in different components of $G[Q_i]-Art(G[Q_i])$.

\medskip\noindent
{\em Proof of the claim.} We know that every vertex of $Q_i$ is adjacent to at most one articulation point of $G[Q_i]-Art(G[Q_i])$, this $G[Q_i]-Art(G[Q_i])$ has two components, and the vertices of each component are adjacent to the same articulation point (if they are adjacent to any articulation point at all). If both $N_{Q_i}(x)$ and $N_{Q_i}(y)$ were disjoint with the same component of $G[Q_i]-Art(G[Q_i])$, then the articulation point adjacent to this component would be an articulation point of entire $G$, contradicting the 2-connectivity of $G$. \qed

\medskip\noindent
{\em Claim 7.3} \ Suppose that both $G[Q_1]$ and $G[Q_2]$ are 2-connected. Then $G$ has a Hamiltonian path.

\medskip\noindent
{\em Proof of the claim.} Consider $Q_1$. Since $G[Q_1]$ is 2-connected, conditions (a) and (b) of Theorem \ref{thm:3K1_uv} are satisfied for any choice of $a \in N_{Q_1}(x), b \in N_{Q_1}(y) $. If, for some choice of $a\neq b$, condition (c) is also satisfied, we construct a Hamiltonian path in $G$ as it was described in the beginning of Case~2a. Hence suppose that for any pair of distinct vertices $a \in N_{Q_1}(x), b \in N_{Q_1}(y) $, $\{a,b\}$ is a minimum vertex cut in $G[Q_2]$.
If $|N_{Q_1}(x) \cup N_{Q_1}(y)| = 2$, then $N_{Q_1}(x) \cup N_{Q_1}(y)$ is a vertex cut of size 2 in $G[Q_1]$, and then it is also a vertex cut of size 2 in $G$ and after removing it, we get at least 3 components. This would violate condition $(d)$ of this theorem. 

Suppose that $|N_{Q_1}(x) \cup N_{Q_1}(y)| > 2$ and suppose that $x$ has at least 2 neighbours in $G[Q_1]$. It follows from observations done in Case~2c of Theorem~\ref{thm:4k1_1con} and Figure~\ref{fig:Q_2} that $G[Q_1 \cup \{x\}]$ has a Hamiltonian cycle. Thus, if $G[Q_2]$ has a Hamiltonian path starting in a neighbour of $y$, then we are done. If this not the case, by Theorem~\ref{thm:3K1u}, we have that $y$ is adjacent only to articulation points of $G[Q_2]$ which contradicts our assumption that $G[Q_2]$ is also 2-connected. \qed

\medskip
Let us now summarize the Case~2a. Suppose first that at least one of the components $Q_1, Q_2$ induces a graph which is not 2-connected, let it be $Q_1$. If all neighbours of  $x$ or $y$ in $Q_1$ are articulation points of this components, then $G$ contains a Hamiltonian path by Claim~7.1. If neither $N_{Q_1}(x)$ nor $N_{Q_1}(y)$ is a subset of $Art(G[Q_1])$, then by Claim~7.2 we can  choose a neighbour $a$ of $x$ and a neighbour $b$ of $y$, both in $Q_1$, in such a way that they belong to different components of $G[Q_1]-Art(G[Q_1])$. It follows that for any articulation point $u\in Art(G[Q_1])$, $a$ and $b$ belong to different components of $G[Q_1]-\{u\}$ and hence condition (b) of Theorem~\ref{thm:3K1_uv} is satisfied. Condition (a) is satisfied because $\{a,b\}\cap Art(G[Q_1]) = \emptyset$ and condition (c) is satisfied because $G[Q_1]$ is not 2-connected. Thus a Hamiltonian path from $a$ to $b$ in $G[Q_1]$ exists, and it can be extended to a Hamiltonian path in $G$ by a path cover $\mathrm{PC}(a',b')$ in $G[Q_2]$, with $a'\in N_{Q_2}(x)$, $b'\in N_{Q_2}(b')$, $a'\neq b'$. Finally, if both $G[Q_1]$ and $G[Q_2]$ are 2-connected, $G$ has a Hamiltonian path by Claim~7.3.
 
\textbf{Case 2b.} \textit{ One of the components, say $Q_1$, is a clique and the other one, say $Q_2$, induces a $4K_1$-free graph.} 

Observe that if there exits a path cover of size 2 in $G[Q_2]$ with one path starting at a vertex in $N_{Q_2}(x)$ and the other one 
starting in a vertex in $ N_{Q_2}(y)$, then $G$ has a Hamiltonian path. Such a path cover does not exist if and only if
at least one condition of Theorem \ref{thm:4k1pc} is violated for $G[Q_2]$. Several cases can occur. 

\textit{Case A. There is an articulation point $z$ of $G[Q_2]$, such that $G[Q_2] -\{z\}$ has at least 3 components.}

Let $z$ be such an articulation point of $G[Q_2]$, i.e., $G[Q_2] -\{z\}$ has at least 3 components, let us denote them $J_1, J_2, J_3$. Since $G[Q_2]$ is $4K1$-free, all of $G[J_1], G[J_2], G[J_3]$ induce cliques. First suppose that (a) of Theorem \ref{thm:4k1pc} is violated for any pair of distinct vertices, one from $N_{Q_2}(x)$ and second one from $N_{Q_2}(y)$. If $N_{Q_2}(x)$ (or $N_{Q_2}(y)$)  contains only $z$, then $c(G-\{y,z\}) \geq 3$ (or $c(G-\{x,z\}) \geq 3$) which would violate condition (d) of this theorem. Also,  $N_{Q_2}(x) \cup N_{Q_2}(y)$ cannot be a subset of $(G[J_i] \cup \{z\})$ for any $i=1,2,3$, since $G$ is 2-connected. Thus, this is not possible. 

Assume that $x$ and $y$ have neighbours in different components of $G[Q_2] -\{z\}$, say $x$ has a neighbour $u$ in $J_1$ and $y$ has a neighbour $v$ in $J_2$ (thus, (a) of Theorem \ref{thm:4k1pc} is satisfied for $u,v$). It follows from the structure of $G[Q_2]$ that 
$G[Q_2]$ does not have 3 articulation points inducing a triangle, and hence condition (b) of Theorem \ref{thm:4k1pc} must be satisfied for $u,v$. 

Suppose that (c) of Theorem \ref{thm:4k1pc} is violated for $u,v$. This means that $\{u,v\}$ is a vertex cut in $G[Q_2]$ such that $c(G[Q_2] - \{u,v\} ) \geq 3$. This is possible only if $u$ is the only neighbour of $z$ in $J_1$ and $v$ is the only neighbour of $z$ in $J_2$ and $|J_1|, |J_2| >1$. Since $G[Q_2]$ is $4K_1$-free, $z$ is adjacent to all vertices in $J_3$. Since $G$ is 2-connected, then $x$ or $y$ must have a neighbour in $J_3$, say $x$ has a neighbour $u'$ in $J_3$. Also note that $u'$ is not an articulation point of $G[Q_2]$. 

Now consider the pair $u', v$. We can obtain a path cover $\mathrm{PC}(u',v)$ by taking one path starting at $u'$, then traversing all vertices of $J_3$, $z$, $u$ and all vertices of $J_1$. The second path starts in $v$ and traverses all vertices of $J_2$.

\textit{Case B. There exist three articulation points $a,b,c$ inducing a triangle in $G[Q_2]$.}
Let $a,b,c$ be three articulation points inducing a triangle in $G[Q_2]$. Denote the components of $G[Q_2]-\{a,b,c\}$ by $J_a, J_b, J_c$ correspondingly. Since, $G[Q_2]$ is $4K_1$-free, $a$ is adjacent to all vertices in $J_a$, $b$ is adjacent to all vertices in $J_b$ and $c$ is adjacent to all vertices in $J_c$.

First suppose that (b) of Theorem~\ref{thm:4k1pc} is violated for any pair of distinct vertices, one from $N_{Q_2}(x)$ and second one from $N_{Q_2}(y)$.
Since $G$ is 2-connected, then $x$ or $y$ has a neighbour in $G[Q_2]$ which is not an articulation point and $N_{Q_2}(x) \cup N_{Q_2}(y)$ cannot be a subset of vertices $\{a,b,c\}$ together with vertices from one component of $G[Q_2]-\{a,b,c\}$. Thus, this is not possible. 

It follows that at least one of $x, y$ has a neighbour which is not articulation point of $G[Q_2]$, say $x$ has a neighbour $u$ in $J_a$, and $y$ has a neighbour in 
$G[Q_2 - (J_a \cup \{a \})]$, say $y$ is adjacent to $v$ in $J_b \cup \{b\}$. We can obtain one path starting in $x$, going through $u$, all vertices of $J_a$, $a$, $c$ and all vertices of $J_c$. The second path starts at $y$, goes through $b$ and all vertices of $J_b$. 

\textit{Case C. For a pair of distinct vertices $u \in N_{Q_2}(x)$ and $v \in N_{Q_2}(y)$, $\{u,v\}$ is a vertex cut (not necessarily minimal) in $G[Q_2]$ such that $c(G[Q_2 - \{u,v\}]) \geq 3$. }

First suppose that $|N_{Q_2}(x) \cup N_{Q_2}(y)| = 2$. We have that for a pair of distinct vertices $u \in N_{Q_2}(x)$ and $v \in N_{Q_2}(y)$, $\{u,v\}$ is a vertex cut (not necessarily minimal) in $G[Q_2]$ such that $c(G[Q_2 - \{u,v\}]) \geq 3$. Then $\{u,v\}$ is also a vertex cut of size 2 in $G$ and $c(G - \{u,v\}) \geq 4$ which would violate condition (d) of this theorem. 

Now suppose that $|N_{Q_2}(x) \cup N_{Q_2}(y)| > 2$. Let $u,v$ be fixed pair of distinct vertices, $u \in N_{Q_2}(x)$ and $v \in N_{Q_2}(y)$ such that $\{u,v\}$ is a vertex cut (not necessarily minimal) in $G[Q_2]$ and $c(G[Q_2 - \{u,v\}]) \geq 3$. 
We will distinguish two cases. 

\textit{Case C.1. The pair $\{u,v\}$ is a minimal vertex cut in $G[Q_2]$.} Since $G[Q_2]$ is $4K_1$-free graph, $G[Q_2 - \{u,v\}]$ has exactly 3 components, denote them $J_1, J_2, J_3$ and they are cliques. Each of $u,v$ has at least one neighbour in each component. 

Suppose that $G[Q_2]$ is not 2-connected. There is an articulation point $z$ of $G[Q_2]$. From the minimality of the vertex cut $\{u,v\}$ we have that $z \neq u,v$. Since $J_1,J_2,J_3$ are cliques, the only possibility is that $z$ is the only neighbour of $u$ and $v$ in one of the cliques, say $J_1$, and the size of $J_1$ is at least 2. Also note that $z$ is adjacent to all  other vertices of $J_1$. The graph $G$ is 2-connected, thus $x$ or $y$, say $x$, has a neighbour $x'$ in $G[J_1 - \{z\}]$. We can obtain one path starting at vertex $x'$, traversing through all vertices of $J_1$, $z$, $u$ and all vertices of $J_2$. Second path goes through $v$ and all vertices of $J_3$.

Suppose that $G[Q_2]$ is 2-connected. Since $|N_{Q_2}(x) \cup N_{Q_2}(y)| > 2$, at least one of $x, y$ has a neighbour in $G[Q_2]$ other than $u$ and $v$, say $x$ has a neighbour $u'$ in $J_1$. The pair $\{u,v\}$ is a minimal vertex cut in $G[Q_2]$ and $G[Q_2]$ is 2-connected, thus, $u$ and $v$ have together two distinct neighbours in each component.  We can obtain one path starting at vertex $u'$ and traversing all vertices of $J_1$. Second path starts at $v$ and goes through all vertices of $J_2$, $u$ and all vertices of $J_3$.

\textit{Case C.2. The pair $\{u,v\}$ is a vertex cut in $G[Q_2]$ such that $c(G[Q_2 - \{u,v\}]) \geq 3$ and it is not minimal.} Thus, $G[Q_2]$ is not 2-connected.

If there is an articulation point $z$ of $G[Q_2]$ such that $c(G[Q_2 - \{z\}]) = 3$ (it is possible $z=u$ or $z=v$), then we are in Case $A$. 

Thus, suppose that for each articulation point of $G[Q_2]$, $G[Q_2]$ has 2 components after removing it. Since $\{u,v\}$ is a not minimal vertex cut in $G[Q_2]$, at least one of $u, v$ is an articulation point of $G[Q_2]$, say $u$. The graph $G[Q_2 - \{u\}]$ has 2 components, denote the one containing $v$ by $J_1$ and the other one by $J_2$. The vertex $v$ must be an articulation point in $G[J_1]$ and thus, $G[J_2]$ is a clique. Denote the components of $G[J_1 - \{v\}]$ by $R_1, R_2$. Let us summarise the situation. The graph $G[Q_2 - \{u,v\}]$ has 3 components $J_2, R_1, R_2$, all of them are cliques, $u$ is adjacent to $J_2$ and $J_1 = R_1 \cup \{v\} \cup R_2$. The vertex $v$ is adjacent to $R_1$ and $R_2$. 

If $v$ is the only neighbour of $u$ in $G[J_1]$, then $c(G[Q_2 - \{v\}]) = 3$ and we are again in Case A. Thus, we can assume without loss of generality that $u$ has a neighbour $u'$ in $R_1$. Since $G$ is 2-connected, there exists a vertex $w$ in $G[J_2]$ which is not the only neighbour (if $|J_2|>1$) of $u$ in $G[J_2]$ such that $w$ is adjacent to $x$ or $y$.

If $w$ is a neighbour of $x$, then we can obtain one path starting at $w$, then going through all vertices of $J_2$, $u$ and all vertices of $R_1$. The second path starts at $v$ and goes through all vertices of $R_2$.

Suppose that $w$ is a neighbour of $y$. If all neighbours of $u$ in $J_1$ are articulation points of $G[J_1]$, then we are in Case A or Case B. Otherwise, we can obtain one path going through $w$ and all vertices of $J_2$. Second path goes through $u$ and all vertices of $J_1$ (existence of such path is guaranteed by Theorem~\ref{thm:3K1u}).

\textit{Case D. There is a pair of distinct vertices $u \in N_{Q_2}(x)$ and $v \in N_{Q_2}(y)$, such that condition (d) of Theorem \ref{thm:4k1pc} is violated for $u,v$}.

Thus, we have the following situation. 
There is an articulation point $w$ of $G[Q_2]$ such that $c(G[Q_2]-\{w\}) = 2$, let $J_1$, $J_2$ be the components of $G-\{w\}$. Suppose that $G[J_1]$ contains an articulation point $z$ and denote the component of $G[J_1-\{z\}]$ containing $u$ by $R_1$ and the other component by $R_2$. Observe that $R_1, R_2, J_2$ are cliques.

Since condition (d) of Theorem \ref{thm:4k1pc} is violated for $u,v$,  
we have that $v \in R_1$, $|R_1| > 2$, $w$ is adjacent only to vertices from $R_1 \cup \{z\}$ and there is no vertex $s \neq u,v$ adjacent to $w$ or $z$. Since $G$ is 2-connected, there exists a vertex $w'$ in $J_2$ such that $w'$ is adjacent to $x$ or $y$ and $w'$ is not an articulation point of $G[Q_2]$, i.e. either $|J_2|=1$ or $w$ has a neighbour in $J_2$ other than $w'$. 

Suppose that $w'$ is adjacent to $x$. If $v$ is not an articulation point of $J_1$, then we can take one path going through $w'$, all vertices of $J_2$ and $w$. As the second path we can take a Hamiltonian path of $G[J_1]$ starting at $v$ (it exists by Theorem~\ref{thm:3K1u}). 
If $v$ is an articulation point of $G[J_1]$ and it is the only neighbour of $w$ in $J_1$, then we are in Case A. Thus, assume that $w$ has a neighbour other than $v$, i.e. $w$ is adjacent to $u$. In this case, we can obtain one path going through $ w', w, u$ and all vertices of $R_1$ and second path going through $ v, z$ and all vertices of $R_2$. 

Analogously, if $w'$ is adjacent to $y$ (again, we will have two case depending on whether $u$ is an articulation point of $G[J_1]$ or not).

\textit{Case E. None of the cases A.-D. occurred.} 

In this case, there is no possibility for violation of conditions of Theorem~\ref{fig:4k1pathcover}. Thus, there exists a sought path cover.

\textbf{Case 3.} \textit{ $G$ is at least 3-connected.} 

In this case, $G$ has a Hamiltonian path by Proposition~\ref{prop:chvatal}.1. 
\end{proof}

\section{Conclusion}\label{sec:Conclusion}

We characterized the structure of graphs having a Hamiltonian path and path cover of size two with prescribed endpoints for the class of $3K_1$-free and $4K_1$-free graphs. The characterization was then applied to characterize $5K_1$-free graphs with a Hamiltonian path. All the characterizations give explicit algorithms for checking the existence of a Hamiltonian path, cycle, and connectedness. We leave structural description of $kK_1$-free graphs having Hamiltonian path, cycle or being Hamiltonian-connected for higher $k$ as a future research.

\bibliographystyle{splncs04}
\bibliography{references}
\end{document}